\documentclass[11pt]{amsart}

\usepackage{amsmath, amssymb}
\usepackage{amsfonts}
\usepackage{amsthm}
\usepackage{amsopn}
\usepackage{amscd}

\usepackage [all]{xy}

\newcommand{\sO}{\mathcal{O}}
\newcommand{\sC}{\mathcal{C}}
\newcommand{\cinf}{\sC^\infty}
\newcommand{\sj}{\mathcal{J}}
\newcommand{\sT}{\mathcal{T}}

\newcommand{\sD}{\mathcal{D}}

\newcommand{\sX}{\mathcal{X}}
\newcommand{\sman}{\mathbf{SM}}
\newcommand{\alg}{\mathbf{Alg}}
\newcommand{\sm}{\sman}
\newcommand{\bR}{\mathbf{R}}

\newcommand{\bZ}{\mathbf{Z}}
\newcommand{\A}{\mathbb{A}}

\newcommand{\sh}{\sharp}
\newcommand{\ruu}{\bR^{1|1}}
\newcommand{\rou}{\bR^{0|1}}
\newcommand{\rpq}{\bR^{p|q}}
\newcommand{\tensor}{\otimes}
\newcommand{\comp}{\circ}

\newcommand{\iso}{\cong}
\newcommand{\tic}{\tilde{c}}

\theoremstyle{plain}
\newtheorem{thm}{Theorem}[section]
\newtheorem{prop}[thm]{Proposition}

\newtheorem{cor}[thm]{Corollary}

\newtheorem{lem}[thm]{Lemma}

\newtheorem{rem}[thm]{Remark}

\theoremstyle{definition}

\theoremstyle{remark}

\let\sym Q
\let\str D
\let\sh\sharp
\let\over\stackrel
\let\a\alpha

\let\g\gamma

\let\e\varepsilon

\let\th\theta

\let\l\lambda
\let\m\mu
\let\n\nu

\let\f\varphi
\let\o\omega

\let\G\Gamma
\let\D\Delta

\let\La\Lambda

\let\O\Omega

\let\na\nabla
\let\ra\rightarrow
\let\lra\longrightarrow
\let\ti\tilde

\newcommand{\vq}{\partial_\th-\th\partial_t}
\newcommand{\vd}{\partial_\th+\th\partial_t}

\begin{document}
\title{Superconnections and Parallel Transport}
\author{Florin Dumitrescu}
\date{\today}
\maketitle

\begin{abstract}
This note addresses the construction of a notion of parallel
transport along superpaths arising from the concept of a
superconnection on a vector bundle over a manifold $M$. A
superpath in $M$ is, loosely speaking, a path in $M$ together with
an odd vector field in $M$ along the path. We also develop a
notion of parallel transport associated with a connection (a.k.a.
covariant derivative) on a vector bundle over a
\emph{supermanifold} which is a direct generalization of the
classical notion of parallel transport for connections over
manifolds.
\end{abstract}

\section{Introduction}

The problem of understanding geometrically superconnections arose
in an attempt to give examples of supersymmetric $1|1$- field
theories $\grave{a}$ la Stolz-Teichner (see \cite{ST}) over a
manifold. Such field theories are expected to constitute cocycles
for a differential version of topological K-theory (see \cite{HS}
and \cite{BuS}), providing the appropriate frame to formulate for
example local family versions of the celebrated index theorems.
Though these problems are related, it is not clear at this point
how to formulate this connection.

Recall the classical parallel transport of a connection on a
vector bundle over a manifold. Let $E$ be a vector bundle over a
manifold $M$, and $\na$ a connection on $E$. Given a path $\g:
[0,1]\ra M$ joining two points $x$ and $y$ in $M$, the connection
allows to identify linearly the fiber $E_x$ over $x$ with the
fiber $E_y$ over $y$, via parallel sections along the path, i.e.
sections that are constant along the path with respect to the
pull-back connection. These identifications are compatible with
gluing paths and invariant under reparametrization. This
independence on metric allows us to refer to connections as
examples of topological 1-dimensional field theories over a
manifold.

An analogous construction carries over to the category of vector
bundles over supermanifolds. The peculiar feature of a
supermanifold is that its functions can also anticommute. Locally,
a supermanifold of dimension $(p,q)$ looks like $\rpq$, the space
whose functions are smooth functions on $\bR^p$ tensored with an
exterior algebra in $q$ odd generators. A supermanifold is
represented by its ($\bZ/2$-graded) algebra of functions. For
example, given a vector bundle $E$ over a manifold, the
supermanifold $\Pi E$ has as functions the sections of the
exterior bundle $\La E^*$. In particular, the ``odd tangent
bundle" $\Pi TM$ of a manifold $M$ has as functions the sections
of the bundle $\La T^*M$, i.e. differential forms on $M$.

A connection $\na$ on a vector bundle $E$ over a supermanifold $M$
is defined by the usual Leibniz property: $\ \na(fs)= df\tensor
s\pm f\na s$, for functions $f$ on $M$, and sections $s$ of $E$.
Geometrically, such a connection gives rise to parallel transport
along superpaths $\a: \ruu \ra M$. Namely, a section along $\a$ is
parallel if it is constant in the direction of the \emph{odd}
vector field $D= \vd$ on $\ruu$. The resulting parallel transport
by parallel sections along superpaths is compatible under glueing
of superpaths and (conformal- see Section \ref{reparametrisation})
reparametrizations of superpaths.

Next, we want to interpret geometrically (as parallel transport)
superconnections. Recall (see \cite{Q} or \cite{BGV}) that a
superconnection on a $\bZ/2$-graded vector bundle $E$ over a
manifold $M$ is an odd degree first order differential operator
defined on the space $\O^*(M, E)$ of sections of the bundle
$\Lambda T^*M\tensor E$ over $M$, $ \A: \O^*(M, E)\ra \O^*(M, E)\
$ satisfying the (graded) Leibniz rule. $\A$ can be written $ \A=
\A_0+ \A_1+ \A_2+\ldots,\ $ with $\A_1$ a grading preserving
connection on $E$, and $\A_i$ for $i\neq 1$ is given by
multiplication by some form $\o_i\in\O^i(M, End\ E)$.

We begin with a notion of parallel transport associated to a
grading preserving connection $\na$ on a $\bZ/2$-graded vector
bundle $E$ over a manifold $M$ and an $End\ E$-valued form
$A\in\O^*(M, End\ E)^{odd}$ on $M$. Then $A$ can be viewed as a
section of the endomorphism bundle $End\ \pi^*E$, where $\pi: \Pi
TM \ra M$ is the map which on functions is the inclusion of
0-forms into the space of differential forms. Let $c: \ruu \ra M$
be a superpath in $M$. (We think, as it is customary, in families
of superpaths- see \cite{W3} or \cite{Fr}.) Then $c$ lifts
naturally (see Section \ref{prel}) to a superpath $\tic :\ruu \ra
\Pi TM$. A section $\psi$ along $c$ is \emph{parallel} if
\[ (c^*\na)_D \psi -(\tic^*A)\psi =0 . \]
The resulting parallel transport is compatible under gluing of
superpaths and converges (by an inverse adiabatic limit process)
to the parallel transport associated to the connection $\na$. As a
corollary to this construction we obtain a parallel transport
corresponding to a superconnection $\A= \A_1+ A$, by viewing
$\A_1$ as a graded connection and $A\in\O^*(M, End\ E)^{odd}$ as
above.

I would like to thank especially Stephan Stolz for guidance and
invaluable help, Nigel Higson for continuous encouragement and
Liviu Nicolaescu for stimulating discussions. This paper was
prepared, in part, while the author was a Postdoctoral Scholar at
Pennsylvania State University.

\section{A short introduction to supermanifolds}
We give in this section a brief introduction to the theory of
supermanifolds. The subject was introduced and developed by Leites
\cite{Le}, Bernstein, Manin \cite{Ma}. A good expository reference
is Deligne and Morgan \cite{DM}. The reader can also consult
Varadarajan \cite{Va}. We spend here some time talking about
integration of vector fields on supermanifolds, since we could not
find a detailed account on the topic in the literature.

Start with the ringed space $\rpq=(\bR ^p,\sC ^\infty_{\bR ^p}
\tensor \Lambda[\theta_1, \ldots, \theta_q])$. A \emph
{supermanifold} $M$ of dimension $p|q$ is a pair $(|M|, \sO _M)$
with $|M|$ a topological space and $\sO _M$  a sheaf on $|M|$ of
$\mathbf{Z}/2$-graded algebras that locally is isomorphic to
$\rpq$. $|M|$ is called the \emph {underlying space} of $M$ and
$\sO _M$ is the \emph {structure sheaf} of $M$. The odd functions
generate a nilpotent ideal $\sj$ of $\sO _M$ and $(|M|, \sO _M/
\sj )$ is a smooth manifold of dimension $p$, called the \emph
{reduced} manifold $M_{red}$ of $M$.

A \emph {morphism of supermanifolds} $f:M \rightarrow N$ is a pair
$f= (|f|, f^\sh)$ consisting of a continuous map $|f|:|M|
\rightarrow |N|$ and a map $f^\sharp :\sO _N \rightarrow |f|_ *
\sO_M$ of sheaves of $\bZ /2$-graded algebras. For example, there
is a canonical morphism $ i:M_{red}\hookrightarrow M$, which on
the underlying spaces is the identity and the map on sheaves is
the projection $i^\sharp: \sO_M\ra \sO_M/\sj$. A morphism $f: M
\rightarrow N$ induces a morphism between the corresponding
reduced manifolds since it preserves the nilpotent ideal sheaves.
A morphism of supermanifolds is uniquely determined by the map
induced on global sections (see \cite{Ko}, p. 208). So, instead of
a map of sheaves, we will consider just the map induced on their
global sections. Supermanifolds are examples of ringed spaces and
the category $\sman$ of supermanifolds embeds fully faithfully in
the category of ringed spaces.

An important source of examples of supermanifolds comes from
vector bundles. To any vector bundle $E$ over a manifold $M_0$ one
can associate a supermanifold $\Pi E=(M_0, \sO_{\Pi E})$ where
$\sO_{\Pi E}$ is the sheaf of sections of $\Lambda E^*$. This
defines a functor
\[ S: \mathbf {VB} \rightarrow \sman: E\mapsto
\Pi E \]
from the category of vector bundles to the category of
supermanifolds. There is also a functor $V$ going the other
direction. Namely, let $M=( M_0,\sO_M)$ be a supermanifold. Then
$(\sj / \sj ^2)^*$ is a locally free sheaf on $M_0$, where $\sj$
is the nilpotent ideal of $\sO_M$, so it determines a vector
bundle on $M_0$. We have that $V\comp S=id$ and $S\comp V=id$ on
isomorphism classes of objects. This doesn't assure an equivalence
of categories though, since $SV$ fails to be the identity on
morphisms (e.g. it maps the automorphism $(x, \th_1, \th_2)
\mapsto (x+\theta_1 \theta_2, \th_1, \th_2)$ of $\bR^{1|2}$ to
$id$). The category of supermanifolds is richer in morphisms. This
relation between the categories is analogous to the one between
graded rings (vector bundles) and filtered rings (supermanifolds).

\subsection{The ``functor of points" viewpoint} In the superworld
one cannot talk properly about points on a supermanifold unless
one refers to points on the reduced manifold. A more suitable
approach is the lingo of \emph{S-points}. Consider $M$ a
supermanifold. An $S$-point of $M$ for an arbitrary supermanifold
$S$ is a map $S\rightarrow M$ and the $S$-points of $M$ is the
\emph{set} $M(S)=\sman(S,M)$. This is the approach physicists
adopt in computations, which also resonates with our geometric
intuition. One can think of an $S$-point as a family of points of
$M$ parametrised by $S$. For example, as sets:
\[ M(\bR^{0|0})= \sman(\bR^{0|0},M)=|M|. \]
If $T\over{\a}\longrightarrow S$ is a map in $\sman$, there is a
natural map $M(S)\rightarrow M(T):m\mapsto m\comp\a$. So $M$
determines a contravariant functor:
\[ \sman^{op}\rightarrow \mathbf{Sets} : S\mapsto M(S) \]
called the \emph{functor of points} of $M$. A map $f:M\rightarrow
N$ of supermanifolds determines a natural transformation $\sman(\
\cdot\ ,M)\rightarrow \sman(\ \cdot\ ,N)$. The converse of this is
also true, and forms the content of Yoneda's lemma. This means
that to give a map $M\rightarrow N$ amounts to giving maps of sets
$M(S)\rightarrow N(S)$, natural in $S$.

One can therefore think of a supermanifold $M$ as a representable
functor $\ \sman^{op}\rightarrow \textbf{Sets}$, such a functor
determining $M$ uniquely up to isomorphism. For example, if $M, N$
are two supermanifolds, their \emph{product} $M\times N$ can be
interpreted as the supermanifold representing the functor
\[ S\mapsto \sman(S, M)\times \sman(S, N). \]
An arbitrary contravariant functor $\sman\rightarrow \mathbf{Sets}
$ will be called a \emph{generalized supermanifold}. The category
$\sman$ of supermanifolds embeds fully faithful into the category
$\mathbf{GSM}$ of generalized supermanifolds. Consider, for
example, two supermanifolds $M, N$ and define the generalized
supermanifold
\[ \underline{\sman}(M, N):\ \sman\rightarrow \mathbf{Sets}: S\longmapsto \sman(S\times M, N).\]
If $\underline{\sman}(M, N)$ is an ordinary supermanifold, then we
have the following adjunction formula
\[ \sman(S,\ \underline{\sman}(M, N)) \cong\sman(S\times M, N). \]

\subsection{The tangent sheaf and tangent vectors.} The analogue of
the tangent bundle in classical differential geometry is the
\emph{tangent sheaf} $\sT M$ defined as the sheaf of graded
derivations of $\sO_M$, i.e. for $U\subseteq |M|$
\begin{displaymath}
\sT M(U)=\{ X:\sO_M(U)\rightarrow\sO_M(U) \mbox{ linear }|
X(fg)=X(f)g+(-1)^{p(X)p(f)}fX(g)\}.
\end{displaymath}
Here $p(X)= 0$ or $1$ according to whether $X$ is even,
respectively odd vector field on $U$, and similarly $p(f)=0$ or 1,
for $f$ even, respectively odd, function on $M$. $\sT M$ is then a
locally free $\sO_M$-module of rank $(p,q)$ the dimension of the
supermanifold $M$. Sections of $\sT M$ are the \emph{vector
fields} on $M$. For $X$ and $Y$ vector fields on $M$, define as
usual their \emph{Lie bracket} $[X, Y]$ by
\[ [X, Y](f)= X(Y(f))-(-1)^{p(X)p(Y)}Y(X(f)), \text{ for }\ \
f\in\cinf(M)= \sO_M(|M|). \] For example, consider on $\ruu$ the
vector field $D=\partial_\th+\th\partial_t$. Then, one can check
that
\[ D^2=\frac{1}{2}[D, D]=\partial_t. \]
Similarly, if $Q= \partial_\th-\th\partial_t$, then
\[ Q^2=\frac{1}{2}[Q, Q]= -\partial_t. \]

For $m\in M(S)$ an $S$-point of $M$, define the tangent space at
$m$ to $M$ by
\[ TM_m= \{ v: \cinf(M)\ra \cinf(S) |\
v(fg)=v(f)m^\sharp(g)+(-1)^{p(v)p(f)}m^\sharp(f)v(g)\}. \] For
$m\in M$ an ordinary point, we get the usual definition of the
tangent space at $m$.

\subsection{Geometric structures on $(1,1)$-supermanifolds} \label{metric}
Let $Y$ be a $(1,1)$-supermanifold. Then, the tangent sheaf $\sT
Y$ is a locally free $\sO_Y$-module of rank (1,1): if $(t,\th)$
are local coordinates on $Y$ then $\{\partial_t,\ \partial_\th\}$
form a local basis for $\sT Y$. A \emph{conformal structure} on
$Y$ is a rank (0,1)-distribution $\sD$, i.e. a rank (0,1) subsheaf
of the tangent sheaf $\sT Y$, that fits into the following short
exact sequence of sheaves
\[ 0\ra\sD\longrightarrow \sT Y \longrightarrow \sD^{\tensor 2}\ra
0. \]

A \emph{euclidean (metric) structure} on $Y$ is given by the
choice of an odd vector field $D$ generating an odd distribution
$\sD$ as above. For example, on $\ruu$ consider the vector field
$D=\vd$. Then $D$ defines a metric structure on $\ruu$, called the
\emph{standard} metric structure on $\ruu$. Also $\sD= <D>$, the
distribution generated by $D$, defines a conformal structure on
$\ruu$: indeed, the square of $D$ is $D^2=\partial_t$, and the
pair $\{D,\ D^2\}$ generates $\sT \ruu$ as an $\sO_{\ruu}$-module.
For an alternative definition of metric structures see \cite{ST},
section 3.2.

\subsection{The super Lie group $\ruu$} Super Lie groups are the
super analogue of Lie groups in differential geometry. Let for
example $\ruu$ be the super Lie group with the following
multiplication map $m:\ruu\times\ruu\ra \ruu$, defined on
$S$-points by
\[ (t,\th),\ (t',\th')\over{m_S}\longmapsto (t+t'+\th\th',
\th+\th'). \] Here $t$ and $t'$ are even functions on $S$, $\th$
and $\th'$ are odd functions on $S$, and so on... Observe that
$\th\th'$ is an even function on $S$. The map $m$ defines a group
multiplication on $\ruu$, with identity given by $(0,0)\in\ruu$
and the inverse map given by $(t,\th)\mapsto (-t,-\th)$. This is
the group structure on $\ruu$ that we will mostly use in this
paper; therefore we are going to call it the \emph{standard} group
structure on $\ruu$.

As in the classical theory of Lie groups, we can consider left
(right) invariant vector fields and identify them with the tangent
space at the identity $e\in G$. Let $X$ be a vector field on a
super Lie group $G$, i.e. a graded derivation $X:\cinf(G)
\rightarrow \cinf(G)$. $X$ is \emph{left-invariant} if the
following diagram commutes:
\[
\xymatrix{ \cinf(G) \ar[r]^-{m^\sharp} \ar[d]_X
& \cinf(G)\otimes\cinf(G) \ar[d]^{1\otimes X} \\
\cinf(G) \ar[r]_-{m^\sharp} & \cinf(G)\otimes\cinf(G)}
\]
that is: $m^\sharp \circ X= (1 \otimes X)\circ m^\sharp$. The
diagram expresses the fact that $X$ is an infinitesimal right
translation.

Consider, for example, $\ruu$ with the standard group structure
defined above. Let $Q$ be the vector field on $\ruu$ given by
$Q=\partial_\th-\th\partial_t$, in coordinates $(t,\th)$ on
$\ruu$. Let us show that $Q$ is left-invariant. We need to check
that the following diagram is commutative
\[ \xymatrix{ \cinf(\ruu) \ar[r]^-{m^\sh} \ar[d]_Q &
\cinf(\ruu)\tensor\cinf(\ruu) \ar[d]^{1\tensor Q}\\
\cinf(\ruu) \ar[r]_-{m^\sh} & \cinf(\ruu)\tensor\cinf(\ruu).} \]
This is verified by looking at the following two commutative
diagrams
\[ \xymatrix{t \ar@{|->}[r]^-{m^\sharp} \ar@{|->}[d]_{Q} & t_1+t_2+\th_1\th_2 \ar@{|->}[d]^{1\tensor Q} & &
\th \ar@{|->}[r] ^-{m^\sharp} \ar@{|->}[d]_{Q}
& \th_1+\th_2 \ar@{|->}[d]^{1\tensor Q}  \\
-\th \ar@{|->}[r] _-{m^\sharp} & -\th_1-\th_2 & & 1 \ar@{|->}[r]
_-{m^\sharp} & 1 .} \]

Analogously, a vector field $X$ on a supermanifold $M$ is
\emph{right-invariant} if $m^\sharp \circ X= (X \otimes 1)\circ
m^\sharp$. One can check for example that the vector field $D=
\partial_\th+\th\partial_t$ is a right invariant vector field on
$\ruu$.

\subsection{Some identifications. } \label{ident}
\begin{lem} \label{otb} Let $M$ be an ordinary manifold. Then, we can identify
\[ \underline{\sman}(\rou, M) \cong \Pi TM, \]
where $\Pi TM$ is the odd tangent bundle of $M$.
\end{lem}

\begin{proof}We want to show that we have isomorphisms
\[ \Psi_S:\sman(S\times\rou, M)\ra \sman(S, \Pi TM), \]

\noindent natural in $S$, where $S$ is an arbitrary supermanifold.
The left hand side is the set of grading preserving maps of
$\bZ/2$-algebras $$\f:\cinf(M)\ra \cinf(S\times\rou)=
\cinf(S)[\th]. $$ If we write $\f(f)=\f_1(f)+\th \f_1(f)$, for
$f\in\cinf(M)$, then the fact that $\f(fg)=\f(f)\f(g)$ is
equivalent to the following two
conditions \\\\
$\left\{ \begin{array}{l} \f_1(fg)=\f_1(f)\f_1(g)
\\\\
\f_2(fg)=\f_2(f)\f_1(g)+(-1)^{p(f)}\f_1(f)\f_2(g)
\end{array} \right.$\\\\
The first condition is equivalent to $\f_1=a^\sharp$, for some
$a:S\ra M$. The second condition tells us that $\f_2$ is an
\emph{odd} tangent vector at $a\in M(S)$, i.e. $\f_2= X_a\in
TM_a$. Therefore the left hand side is
\[ \sman(S\times\rou, M)=\ \{ \text{pairs } (a, X_a)\ |\ a\in
M(S),\ X_a\in TM_a,\ X_a \text{ odd} \}. \]

\noindent The right hand side $\sman(S, \Pi TM)$ is the set of
$\bZ/2$-graded algebra maps $\O^*(M)\ra \cinf(S)$. Such maps are
determined by their restriction to 0-forms (functions) and 1-forms
(more specifically, 1-forms of the type $df$, for $f\in\cinf(M)$).
Define then $\Psi_S(a, X_a)$ to be the map $S\ra \Pi TM$
determined by defining it on functions $f\in\cinf(M)$ by
$a^\sharp(f)\in\cinf(S)$, and on forms $df$ by $X_a(f)$. One can
easily check that $\Psi_S$ is well-defined, bijective, and natural
in $S$.

\end{proof}

Let $T:\Pi TM\times\rou\ra\Pi TM$ be the map which on functions is
given by $\O^*(M)\ni\o\mapsto \o+\th d\o\in \O^*(M)[\th]$.
Consider also the map $\m: \underline{\sm}(\rou, M)\times\rou \ra
\underline{\sm}(\rou, M)$, defined on $S$-points
\[ \sm(S\times\rou, M)\times \sm(S,\rou)\ra \sm(S\times\rou, M) \]
by $(\f,\eta)\mapsto \f\comp(1\times m)\comp(1\times\eta\times
1)\comp(\D\times 1)$, where $m$ is the group composition map on
$\rou$. The maps $T$ and $\m$ define an action of $\rou$ on the
corresponding spaces.

\begin{lem} \label{equiv} The map defined in the previous lemma $$\Psi: \underline{\sm}(\rou, M)\ra \Pi TM$$  is
$\rou$-equivariant.
\end{lem}

\begin{proof}
We want to show that the following diagram is commutative
\[ \xymatrix{ \underline{\sm}(\rou, M)\times\rou \ar[d]_\m \ar[r]_-{\Psi\times
1} & \Pi TM\times\rou \ar[d]^T \\
\underline{\sm}(\rou, M) \ar[r]_\Psi & \Pi TM.} \] We need that,
for each supermanifold $S$, natural in $S$, the following diagram
commutes
\[ \xymatrix{ \sm(S\times\rou, M)\times\sm(S,\rou) \ar[d]_{\m_S} \ar[r]_-{\Psi_S\times
1} & \sm(S,\Pi TM)\times\sm(S,\rou) \ar[d]^{T_S} \\
\sm(S\times\rou, M) \ar[r]_{\Psi_S} & \sm(S,\Pi TM),} \] or, in
terms of functions we need to have
\[ \xymatrix{ \alg(\cinf(M), \cinf(S)[\th])\times\cinf(S)^{odd} \ar[d]_{\m_S} \ar[r]_-{\Psi_S\times
1} & \alg(\O^*(M),\cinf(S))\times\cinf(S)^{odd} \ar[d]^{T_S} \\
\alg(\cinf(M), \cinf(S)[\th]) \ar[r]_{\Psi_S} &
\alg(\O^*(M),\cinf(S)). } \]

For $a\in M(S)$ and $X_a\in TM_a$ denote by $(a,
X_a)\in\alg(\O^*(M),\cinf(S))$ the map determined by $f\mapsto
a^\sharp(f)$ and $df\mapsto X_a(f)$. (Compare the proof of the
previous lemma.) Via the identification
\[ \alg(\O^*(M),\cinf(S))\times\cinf(S)^{odd}=
\alg(\O^*(M)[\th],\cinf(S)), \] the map
\[ T_S: \alg(\O^*(M),\cinf(S))\times\cinf(S)^{odd}\ra \alg(\O^*(M),\cinf(S)) \]
evaluated at
$\f=((a,X_a),\ti{\th})\in\alg(\O^*(M),\cinf(S))\times\cinf(S)^{odd}$
is determined by saying that\\
$f\stackrel{T^\sharp}\longmapsto f+\th df
\stackrel{\f^\sharp}\longmapsto a^\sharp(f)+\ti{\th}X_a(f)=:
b^\sharp(f)\ \ $
and\\\\
$df \over{T^\sh}\longmapsto df \over{\f^\sh}\longmapsto X_a(f)=:
X_b(f),$\\
where $b\in M(S)$ is defined by $b^\sharp(f)=
a^\sharp(f)+\ti{\th}X_a(f)$, for $f\in\cinf(M)$.

On the other hand,
\[ \m_S: \sm(S\times\rou, M)\times\sm(S,\rou) \ra \sm(S\times\rou,
M)\] is defined by
\[ (\a= (a,X_a), \eta)\longmapsto \a\comp(1\times m)\comp(1\times
\eta\times 1)\comp (\D\times 1), \] or, on functions,
$\m_S((a,X_a),\ti{\th})$ is given by
\begin{eqnarray*}
f & \over{\a^\sh}\longmapsto & a^\sh(f)+\th X_a(f) \\
& \over{1\tensor m^\sh}\longmapsto  & a^\sh(f)+(\th_1+\th_2) X_a(f) \\
& \over{1\tensor \n^\sh\tensor 1}\longmapsto &
a^\sh(f)+(\th_1+\ti{\th}) X_a(f) \\
& \over{ \D^\sh\tensor 1}\longmapsto & a^\sh(f)+\ti{\th} X_a(f)+
\th X_a(f)= b^\sh(f)+\th X_b(f).
\end{eqnarray*}

\noindent Therefore we have
\[ \xymatrix{ ((a,X_a),\ \ti{\th}) \ar[d]_{\m_S}
\ar[rr]_-{\Psi_S\times
1} &  &\ \ \ \ \Big\{ (f,df,\th) \ar@{|->}[r] &  (a^\sharp(f), X_a(f),\ \ti{\th})\Big\}  \ar[d]^{T_S} \\
(b,X_b) \ar[rr]_{\Psi_S} & &  (b, X_b)=\Big\{ (f,df) \ar@{|->}[r]
& (a^\sharp(f)+\ti{\th}X_a(f), X_a(f))\Big\}, } \] which verifies
the commutativity of the above diagram. The lemma is proved.

\end{proof}

\subsection{Differential Equations on Supermanifolds.}\label{de}

In what follows we will show that vector fields (even or odd) on
supermanifolds can be integrated. We consider first the
\textbf{even} case.
\begin{lem} Let $X$ be an even vector field on a compact supermanifold
$M$ (i.e. the underlying manifold is compact). Then there exists a
unique map $c:\bR \times M\ra M$ satisfying the following two
conditions: $$\left\{
\begin{array}{l}
\partial_t\comp c^\sharp= c^\sharp\comp X \\
c\big|_{0\times M}=id_M.
\end{array} \right.$$
\end{lem}
\noindent The map $c$ is called the \emph{flow of the vector field
$X$}.

\begin{proof}
The existence and uniqueness of a global solution follows from the
existence and uniqueness of a local solution, since $M$ is
compact. To solve the local problem, we can w.l.o.g. assume that
$M= \rpq$. Let $x^1,\ldots, x^{p+q}$ be the coordinate functions
on $\rpq$, with the first $p$ coordinates even, and the last $q$
odd. We also write $\th_1, \ldots, \th_q$ for the last $q$ odd
coordinates. Let $c^i$ be the image of $x^i$ under the map
$c^\sharp$. Let us write
\[ c^i= \sum c^i_J\th^J, \text{ with } c^i_J\in\cinf(\bR\times\bR^p). \]
The vector field $X$ can be written $X= \sum_1^{p+q}
a_i\partial_{x^i},$ with $a_i$ even, for $i=1,\ldots,p$,
respectively odd, for $i=p+1,\ldots p+q$, functions on $\rpq$. We
further write
\[ a_i= \sum a^i_J\th^J, \text{ with } a^i_J\in\cinf(\bR^p), \]
with some of the $a^i_J$ possibly zero. The first condition above
holds for a map $c:I\times M\ra M$, with $I$ a small neighborhood
of $0$, if and only if it holds when evaluated on the coordinate
functions $x^i$ on $\rpq$. Consequently, we must have that
\[ \partial_t c^i= c^\sharp(a_i). \]
Equivalently, we have
\begin{eqnarray*}
\sum \frac{dc^i_J}{dt}(t, x)\th^J &=& a_i(c(t, x, \th))\\
&=& a_i(\sum_J c_J(t, x)\th^J)\\
&=& a_i(c_0(t, x)+ \sum_{J\neq 0}c_J(t, x)\th^J)\\
&=& a_i(c_0(t, x)) + \sum\frac{\partial a_i}{\partial
x^j}(c_0(t, x))c_J^j(t, x)\th^J+\ldots \\
&=& a_i(c_0(t, x)) + \sum f_J^i(\frac{\partial^La_i}{\partial
x^L}(c_0(t, x)), c_K(t, x)) \th^J,
\end{eqnarray*}
where $f^i_J$ are polynomial functions on some large euclidean
space, $|L|\leq p$, and $|K|\leq q$. The fourth equality comes
from the Taylor expansion for the function $a_i$ around $c_0(t,
x)$. Equating the coefficients of the above relation, we
obtain the system \\\\
$\left\{
\begin{array}{l} \frac{dc^i_0}{dt}(t, x)= a_i(c_0(t, x)), \ \ \
i=1,\ldots, p\\\\
\frac{dc^i_J}{dt}(t, x)= \sum f_J^i(\frac{\partial^La_i}{\partial
x^L}(c_0(t, x)), c_K(t, x)),\ \ \ \ 0\neq |J|\leq q
\end{array} \right.$ \\ \\
We solve first the system of the first $p$ equations to determine
$c_0$, and then the first order system of differential equations
determined by the last $(p+q)(2^q-1)$ equations. The initial
condition of the system is given by the relations
\[ x^i= \sum c^i_J(0, x)\th^J, \ \ \ \ i=1,\ldots, p+q, \]

\noindent which reflect the condition (2) in the statement of the
lemma. By the general theory of systems of differential equations,
the above system admits a unique solution. The lemma follows.
\end{proof}

More generally, given an even vector field $X$ on a supermanifold
$M$, and a parametrising supermanifold $S$, we have a unique
solution $\a:\bR\times S\ra M$ of the system $$\left\{
\begin{array}{l} \partial_t\comp \a^\sharp= \a^\sharp\comp X \\
\a\big|_{0\times S}=f
\end{array} \right.$$
for some initial condition $f:S\ra M$. It is given by
\[ \a= c \comp (1\times f) \]
where $c$ is the flow determined by $X$. The map $\a$ gives us a
family of \emph{integral curves} for the vector field $X$
parametrised by $S$.

Next, we consider the \textbf{odd} case.

\begin{lem} Let $M$ be a compact supermanifold and $X$ be an odd vector field on
$M$. Then there exists a unique map $\a:\ruu\times S\ra M$
satisfying the following two conditions:
$$\left\{
\begin{array}{l} D\comp \a^\sharp=
\a^\sharp\comp X \\
\a\big|_{0\times S}=f
\end{array} \right.$$
for some initial condition $f:S\ra M$. Here $D= \partial_\th +
\th\partial_t$ is as usual.
\end{lem}

\begin{proof} Again, it is enough to solve the problem locally,
for which we can assume that $M= \rpq$. Write $X= \sum a_i
\partial_{x^i} $. Then the first relation on arbitrary functions
$g$ on $\rpq$ gives
\begin{equation} \label{dag}\sum\Big(\frac{\partial
g}{\partial x^i}\comp\a\Big) \frac{\partial \a^i}{\partial\th}+
\sum\th\Big(\frac{\partial g}{\partial x^i}\comp\a\Big)
\frac{\partial \a^i}{\partial t} = \sum\Big(a_i\frac{\partial
g}{\partial x^i}\Big)\comp \a.
\end{equation}

\noindent Let us write $\a= G+\th H$, with $G, H\in \cinf(I\times
S)$, for some $I$ a neighborhood of $0$. Then, by Taylor's
expansion, we have
\[ a_i(\a)= a_i(G)+ \sum_j\th\frac{\partial a_i}{\partial x^j}(G)H^j\]
and (\ref{dag}) becomes
\[ H^i+ \th\Big(\frac{\partial G^i}{\partial t}+\th\frac{\partial H^i}{\partial
t}\Big)= a_i(G)+ \sum_j\th\frac{\partial a_i}{\partial x^j}(G)H^j.
\] This is equivalent to the system
$$\left\{
\begin{array}{l} a_i(G)= H^i\\\\
\frac{\partial G^i}{\partial t}(s,t)= \sum_j\frac{\partial
a_i}{\partial x^j}(G(s,t))H^j(s,t)
\end{array} \right.$$
which gives rise to the system
\begin{equation}
\label{ddag}\frac{\partial G^i}{\partial t}(s,t)=
\sum_j\frac{\partial a_i}{\partial x^j}(G(s,t))a_j(G(s,t)).
\end{equation}
Now, $\sum_j\frac{\partial a}{\partial x^j}a_j$ is an \emph{even}
vector field on $\rpq$, so, by the previous lemma and the ensuing
remark, the system (\ref{ddag}) admits a unique solution once we
know $G(0,s)$, which is given by the initial condition $f:S\ra M$.
The lemma is proved.
\end{proof}

\begin{rem} The flow of an odd (even) vector field defines
actually an $\ruu$-action (respectively an $\bR$-action) on the
(compact) supermanifold.
\end{rem}
Let $X$ be an odd vector field on a
supermanifold $M$, and let $\a:\ruu \times M\ra M$ be the flow of
$X$. By definition, the following diagram is commutative
\[ \xymatrix{\cinf(M) \ar[r] ^-{\a^\sharp} \ar[d]_{X} & \cinf(\ruu \times M) \ar[d]^D\\
\cinf(M) \ar[r] ^-{\a^\sharp} & \cinf(\ruu \times M).} \]

\noindent Let $u:S\ra \ruu \times M$ be an $S$-point of $\ruu
\times M$. Then
\[ u^\sh\comp D\comp\a^\sh= u^\sh\comp\a^\sh\comp X, \]
which is to say that
\[  \a_{*u}(D_u)= X_{\a(u)}, \]
where $\a_*$ is the differential of $\a$. If we denote $u= (t,\th,
x)$, then the equation above can also be written
\[ \partial_D\a(t,\th, x)= X(\a(t,\th, x)). \]
This relation probably justifies our way of looking at a
differential equation as a commutative diagram. (See also
\cite{Sh}.)

Again, let $X$ be an odd vector field on a supermanifold $M$. By
the lemma above, $X$ defines a flow $\a:\ruu\times M\ra M$. Define
the map $\a_0:\bR\times M\ra M$ by $\a_0= \a\comp (i\times 1_M)$,
where $i:\bR\ra\ruu$ is the standard inclusion map. Moreover $i$
is a group homomorphism, if $\bR$ and $\ruu$ are endowed with the
standard group structures. Therefore $\a_0$ defines a flow map.

\begin{lem} The map $\a_0$ is the flow of the even
vector field $X^2$.
\end{lem}

\begin{proof}
Indeed, by definition $\a^\sh\comp X= D\comp\a^\sh$. Therefore
\begin{eqnarray*}
  \a^\sh\comp X^2 &=& D\comp\a^\sh\comp X \\
   &=& D\comp D\comp\a^\sh \\
   &=& \partial_t\comp\a^\sh.
\end{eqnarray*}
Since $\partial_t$ commutes with $i^\sh\tensor 1$, the claim
follows.
\end{proof}

\noindent\emph{Example}: Let $D$ be the usual vector field on
$\ruu$. Then the flow of $D$ is given by the group multiplication
map $m:\ruu\times\ruu \ra \ruu$. To see this, we should verify
that $m$ fits into the diagram

\[ \xymatrix{\cinf(\ruu) \ar[r] ^-{m^\sharp} \ar[d]_{D} & \cinf(\ruu \times \ruu) \ar[d]^{D\tensor 1}\\
\cinf(\ruu) \ar[r] ^-{m^\sharp} & \cinf(\ruu \times \ruu).} \]
This is indeed the case: the diagram expresses the fact $D$ is a
right invariant vector field.

\section{Connections on supermanifolds and their parallel transport}
The purpose of this section is to describe the parallel transport
along superpaths of a connection on a super vector bundle over a
supermanifold. This follows closely the geometric idea of parallel
transport associated to a connection on a vector bundle over a
manifold.

\subsection{Setup}

Let $E$ be a super vector bundle over a supermanifold $M$, and let
$\na$ be a connection on $E$ (see \cite{DM}), i.e. $\na: \G(M,
E)\ra \O^1(M, E)$ such that
\[ \na(fs)= df\tensor s+ f\na s,\ \ \ \ f\in\cinf(M),\ \
s\in\G(M, E). \] In particular, for $X\in \sX(M)$ a vector field
on $M$, we have $\na_X: \G(M, E)\ra \G(M, E)$ with
\[ \na_X(fs)= X(f)s +(-1)^{p(X)p(f)} f\na_Xs. \]
Let $c:S\times\ruu \ra M$ be a (family of) supercurve(s
parametrized by a supermanifold $S$) in $M$. Consider the
pull-back connection $c^*\na$ and the derivation $(c^*\na)_\str:
\G(c^*E)\ra\G(c^*E)$. Here $\str$ is the vector field $\vd$ on
$\ruu$, extended trivially to $S\times\ruu$. An element of
$\G(c^*E)$ is called a \emph{section of $E$ along $c$}. We say
that the section $s$ along $c$ is \emph{parallel} if
$$(c^*\na)_Ds= 0.$$
In local coordinates, we can think of this as being a
\emph{half-order differential equation}. There are two reasons for
that: first, the vector field $D$ squares to the vector field
$\frac{d}{dt}$, second, for $2n$ unknown functions we need as
initial data $n$ values.

\begin{prop} \label{lift}
Let $c:S\times\ruu \ra M$ be a supercurve in the compact
supermanifold $M$ (i.e. the reduced manifold is compact). Let
$\psi_0\in\G(c_{0,0}^*E)$ be a section of $E$ along $c_{0,0}:S\ra
S\times\ruu\ra M$, with the first map the standard inclusion
$i_{0,0}:S \ra S\times\ruu$. Then, there exists a unique parallel
section $\psi$ of $E$ along $c$, such that $\psi(0,0)= \psi_0$.
\end{prop}
\begin{proof}
The fact that $\psi$ extends to all of $S\times \ruu$ is a
standard argument on the flows of vector fields on \emph{compact}
manifolds. The existence (and uniqueness) of $\psi$ is then a
local problem. Let $U\subseteq M$ be a trivializing neighborhood
such that $E_{|U}\cong U\times \rpq$ ($p|q$ is the rank of the
bundle $E$). Then the connection can be written as $\na= d+ A$,
for some $ A\in \O^1(M)\tensor End(\rpq)^{ev}$. The equation
$(c^*\na)_Ds= 0$ with the given initial condition is then
equivalent to the system
\\\\
$\left\{
\begin{array}{l} \frac{\partial\psi}{\partial D}(s,t,\th) + A(s,t,\th)\psi(s,t,
\th)=0\\\\
\psi(s,0,0)=\psi_0(s)
\end{array} \right.$\\\\
where $\psi$ is defined in a neighborhood of $S\hookrightarrow
S\times \ruu$ with values in $\rpq$, and $A:S\times\ruu \ra
End(\rpq) $ is short for $(c^*A)(D)$. If we write\\
$\psi(s,t,\th)= (a^i(s,t)+\th b^i(s,t))_{i=1, \ldots, p+q}$ \\
$A(s,t,\th)= (c^{ij}(s,t)+\th d^{ij}(s,t))_{i,j=1, \ldots,
p+q}$\\
then the system is equivalent to
\\\\
$\left\{
\begin{array}{l} b^i(s,t)= -c^{ij}(s,t)a^j(s,t)\\\\
\frac{da^i}{dt}(s,t)= -\e(c^{ij}(s,t))b^j(s,t)-
d^{ij}(s,t)a^j(s,t)\\\\
a^i(s,0)=\psi^i_0(s)
\end{array} \right.$\\\\
Here $ \e(a)= \left\{
\begin{array}{l} \ a,\ \ \ \ \ \text{ if $a$ is even}\\
-a,\ \ \ \ \text{ if $a$ is odd}
\end{array} \right. $\\
\noindent It is clear that this system admits a unique solution
around $S\times (0,0)$. The proposition is proved.
\end{proof}

\begin{lem} \label{naturality} \textbf{(Naturality in $S$)}
Let $c:S\times \ruu\ra M$ be a supercurve in $M$, and let
$\f:S'\ra S$ be an arbitrary map. Consider the supercurve
$c':S'\times \ruu\ra M$ defined by $c'= c\comp \bar{\f}$, where
$\bar{\f}=\f\times 1_{\ruu}$. If $\psi$ is a parallel section
along $c$, then $\psi\comp\bar{\f}$ is parallel along
$c\comp\bar{\f}$.
\end{lem}
\begin{proof}
$\psi$ is parallel along $c$ if $(c^*\na)_D\psi= 0.$ Let us
observe that $\bar{\f}_*D= D$.

We therefore have
\begin{eqnarray*}
(\bar{\f}^*c^*\na)_D(\bar{\f}^*\psi) &=& \bar{\f}^\sh\big(
(c^*\na)_{\bar{\f}_*D}\psi\big) \\
&=& \bar{\f}^\sh\big((c^*\na)_{D}\psi\big)\\
&=& 0,
\end{eqnarray*}
since $\psi$ is parallel. That is, $\psi\bar{\f}$ is parallel
along $c\bar{\f}$.

\end{proof}

\begin{rem} \label{qtransport} We could as well defined a parallel section
along a supercurve $c$ in $M$ to be a section $s$ along $c$ that
satisfies the equation
\[ (c^*\na)_\sym s= 0, \]
where $\sym= \vq$. Let us call such sections $\sym$-parallel to
distinguish them from the parallel sections defined above. Their
relevance will become clear in Property (3) of Section
\ref{spath}.
\end{rem}

\subsection{Invariance under reparametrization} \label{reparametrisation} The usual parallel
transport is invariant under reparametrization of paths. We will
see in this subsection what that means in the super-context.

Let $c:S\times \ruu \ra M$ be a supercurve in $M$ and let $\psi$
be a parallel section of $E$ along $c$, i.e. $(c^*\na)_D\psi= 0$.
Let $\f$ be a family of diffeomorphisms of $\ruu$ that preserve
the distribution $\mathcal{D}$, parametrized by $S$. In
particular, $\f_*D= bD$, for some $b\in\cinf(S\times \ruu)$. Then
we have
\[ ((c\comp\f)^*\na)_{D}(\psi\comp\f)=
((c^*\na)_{\f_*D}\psi)\comp\f= b\cdot ((c^*\na)_D\psi)\comp\f=0.
\]

Therefore, we conclude that if $\psi$ is a parallel section of $E$
along $c$, then $\psi\comp \f$ is a parallel section of $E$ along
$c\comp \f$. We say that the parallel transport defined by the
connection is \emph{invariant under reparametrization}. (In our
case, ``reparametrization" refers to diffeomorphisms that preserve
a distribution.)

This notion of parallel transport along superpaths generalizes the
usual notion of parallel transport along paths associated with a
connection in the sense that a parallel section in the old sense
is parallel in the new sense, and the new parametrization
invariance is compatible with the parametrization invariance in
the old sense (a detailed discussion can be found in \cite{D},
Section 4.3).

\subsection{Recovering the connection from the super parallel
transport}\label{recover} The next topic we want to address is the
following: Given a connection on a super vector bundle and its
associated parallel transport, how can we recover the connection?
The answer goes as follows.

To give a connection $\na$ on $E$ over $M$ amounts to specifying
for each vector field $X$ on $M$ an $X$-derivation
$\na_X=\tilde{X}: \G(M, E)\ra \G(M, E)$, i.e.
\[ \ti{X}(fs)= X(f)s+(-1)^{p(X)p(f)}f\ti{X}(s), \ \ \
f\in\cinf(M), \ \ s\in\G(M,E), \] such that the correspondence
$X\mapsto \ti{X}$ is $\cinf(M)$-linear.

Let $X$ be an \emph{odd} vector field on $M$, and let $\a=\a_X:
\ruu\times M\ra M$ be the flow of $X$. By definition (see Section
\ref{de}), $X$ fits into the following diagram

\[ \xymatrix{\cinf(M) \ar[r] ^-{\a^\sharp} \ar@{-->}[d]_{X} & \cinf(\ruu\times M) \ar[d]^D\\
\cinf(M) \ar[r] ^-{\a^\sharp} & \cinf(\ruu\times M).} \]

The pullback-connection via the path $\a$ will define, via the
vector field $D$, a $D$-derivation $\tilde{D}$, on the sections of
the pull-back bundle
\[ \G(\ruu\times M, \a^*E)= \cinf(\ruu\times M)\tensor_{\cinf(M)}
\G(M, E). \]
\begin{lem}
$$\tilde{D}= D\tensor 1+ 1\tensor\tilde{X},$$ where the right hand side is
defined by
\[ f\tensor s \mapsto Df\tensor s+ f\tensor \ti{X}s,  \ \ \ f\in\cinf(\ruu\times M), \ \ s\in\G(M,E).
\]
\end{lem}

\begin{proof}
Indeed, both sides are $D$-derivations, and they coincide on
sections of $E$ pulled-back via the map $\a$. To see the latter,
let us write $\a^*s= 1\tensor s$, for $s\in\G(M,E)$. Then
\[ \ti{D}(1\tensor s)= (\a^*\na)_D(\a^*s)=\a^*(\na_Xs)=1\tensor (\ti{X}s)= (D\tensor 1+ 1\tensor \ti{X})(1\tensor s). \]
\end{proof}

Now, the parallel transport depicts in particular the parallel
sections along $c$ in the direction of $D$. That information is
enough to determine $\ti{D}:\G(c^*E)\ra\G(c^*E)$ as a
$D$-derivation. Indeed, locally, if $s_i,\ i=1,\ldots, p+q$ are
linearly independent parallel sections, then any $s\in\G(c^*E)$
can be written $s= \sum f_is_i$, with $f_i\in\cinf(\ruu\times M$).
Then, $\ti{D}(\sum f_is_i)=\sum D(f_i)s_i$. By the Lemma above, we
have in particular $\ti{D}(\a^*s)=\a^*(\ti{X}s)$, for $s\in\G(M,
E)$, and since $\a^*: \G(M, E)\ra \G(\ruu\times M, \a^*E)$ is
injective, knowing $\tilde{D}$, uniquely determines $\tilde{X}$.

Let now $X$ be an \emph{even} vector field on $M$ and let
$\a:\bR\times M\ra M$ be the flow determined by $X$. Let
$\hat{\a}:\ruu\times M\ra M$ be the trivial extension of $\a$,
i.e. $\hat{\a}= \a\comp (p\times 1_M)$, where $p:\ruu\ra\bR$ is
the usual projection (which on functions is the inclusion of
functions on $\bR$ into forms on $\bR$). Then
\[ (\hat{\a}^*\na)_D(\hat{\a}^*s)= \th(\a^*\na)_{\partial_t}(\a^*s)= \th
\a^*(\na_Xs), \] for all sections $s\in\G(M, E)$. Since, as
before, $\a^*$ is injective, the lift of $D$ along $\a$ determines
the lift of $X$ given by the connection.

In this way, via the super parallel transport, we can lift
\emph{all} the vector fields on $M$ to the derivations given by
the connection, in other words, the super parallel transport
\emph{recovers the connection}.
\subsection{Parallel transport along superpaths}
\label{spath} Let $(t,\th)\in\ruu_+(S)$ be an $S$-point of
$\ruu_+$. We define a super-analogue of the interval $I_t= [0,t]$
as follows:

Consider the triplet
\[ \xymatrix{ S \ar@{^{(}->}[r]_-{i_{(0,0)}} & S\times\ruu & S
\ar@{_{(}->}[l]^-{i_{(t,\th)}}}, \] with $i_{(0,0)}(s)= (s,0,0)$
and $i_{(t,\th)}(s)= (s,t(s),\th(s))$. Here $\ruu$ is endowed with
the \emph{standard metric structure} given by the odd vector field
$D=\partial_\eta+\eta\partial_u$ in coordinates $(u,\eta)$ on
$\ruu$ (see Section \ref{metric}). We denote this (family of)
superinterval(s) by $I_{(t,\th)}$.

Let $x$ and $y$ be $S$-points of $M$. A \emph{superpath in $M$
parametrized by the superinterval $I_{(t,\th)}$ and with endpoints
$x$ and $y$} is an equivalence class of supercurves $c: S\times
\ruu \ra M$ with $c\comp i_{0,0}=c(0,0)=x$, respectively $c\comp
i_{t,\th}=c(t,\th)=y$ such that $c\sim c'$ if there exists $\e>0$
such that
\[ c(u,\eta)= c'(u,\eta) \]
for all $(-\eta,0)<(u, \eta)<(t+\eta,\th)$. Here, $``<"$ is a
partially defined order as follows: for $(t,\th),\ (u,
\eta)\in\ruu(S)$, we say
$$(u,\eta)<(t,\th) \ \ \text{ if }\ \
(t,\th)(u, \eta)^{-1}\in\ruu_+(S).$$ Recall that $\ruu$ is a super
Lie group- see Section \ref{liegr}- with the following group
structure
\[ (t,\th),\ (s, \eta)\longmapsto \ (t,\th)(s,\eta):= \ \ (t+s+\th\eta, \th+\eta). \]
In particular, for any supermanifold $S$, $\ruu(S)=\sman(S,\ruu)$
is not just a set but a \emph{group}. $\ruu_+$ is the open
subsupermanifold in $\ruu$ whose reduced part is $\bR_+= (0,\
\infty)$. Such a superpath is denoted for short $c:I_{(t, \th)}
\ra M$.

Let now $c:I_{(t, \th)} \ra M$ be a superpath in $M$. Then the
connection $\na$ on  the bundle $E$ will determine a vector bundle
homomorphism
\[ \xymatrix{ x^*E \ar[dr] \ar[rr]^{SP(c)} & &  y^*E \ar[dl] \\
& S &  } \]

The map $SP(c)$ is given by a $\cinf(S)$-linear map $SP(c): \G(S,
x^*E)\ra \G(S, y^*E)$ described by the following diagram
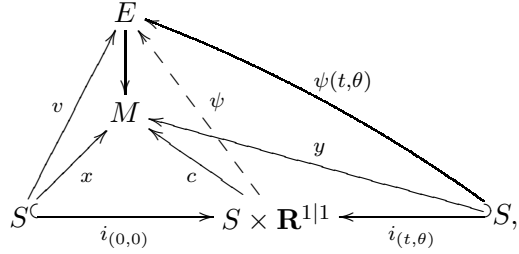
\[ \xymatrix{ & E \ar[d] & & \\
& M & & \\
S \ar[ur]_x \ar[uur]^v \ar@{^{(}->}[rr]_-{i_{(0,0)}} & &
S\times\ruu \ar[ul]^c \ar@{-->}[uul]_\psi & & S
\ar@{_{(}->}[ll]^-{i_{(t,\th)}} \ar[ulll]_y
\ar@/_/[uulll]_{\psi(t,\th)}, }  \] i.e. $SP(c)(v)= \psi(t, \th)$,
where $\psi$ is the unique parallel section of $E$ along the
supercurve $c$, such that $\psi(0,0)= v$. Since the solution
$\psi$ depends on the local data, it turns out that the map
$SP(c)$ is well defined, i.e. it does not depend on a
representative for the superpath $c:I_{(t, \th)} \ra M$. It is
clearly a $\cinf(S)$-linear map, therefore it defines a bundle map
$SP(c):x^*E\ra y^*E$.

The map $SP$ satisfies the usual properties of a parallel
transport map, i.e. it is compatible with gluing superpaths, and
is invariant under reparametrizations (i.e. diffeomorphisms of
superintervals that preserve the fiberwise conformal structure on
$\ruu$). This forms the content of the following

\begin{thm}

Any connection $\na$ on a super vector bundle $E$ over a
supermanifold $M$ gives rise to a correspondence $SP(\na)= SP$
\[ \xymatrix{ I_{t,\th} \ar[r]^c & M & \ar@{|->}[rr]^-{SP(\na)} & & &
c_{0,0}^*E \ar[r] & c_{t,\th}^*E} \] satisfying the following
properties:

\begin{enumerate}
\item The correspondence $c\mapsto SP(c)$ is smooth, and natural
in $S$ (see Lemma \ref{naturality}). Smoothness means the
following: if $c$ is a family of smooth superpaths parametrized by
a supermanifold $S$, then the map $SP(c):c_{0,0}^*E\ra
c_{t,\th}^*E$ is a \emph{smooth} bundle map over $S$.

\item (Compatibility under glueing) If $c:I_{t,\th}\ra M$ and
$c':I_{t',\th'}\ra M$ are two superpaths in $M$ such that
$c'\equiv c\comp R_{t,\th}$ on some neighborhood $S\times U$ of
$S\times (0,0)\hookrightarrow S\times\ruu$, with $U$ an open
subsupermanifold in $\ruu$ containing $(0,0)$, we have
\[ SP(c'\cdot c)= SP(c')\comp SP(c), \]
where $c'\cdot c:I_{t'+t+\th'\th, \th'+\th}\ra M$ is obtained from
$c$ and $c'$ by glueing them along their ``common endpoint", i.e.
\[ \ \ \ \ \ \ \ \ \ \ \ \ \ (c'\cdot c)(s,u,\eta)=
\left\{
\begin{array}{l}  \ \ \ \ \ c(s,u,\eta),\ \ \ \ \ \ \ \ \ \ \ \ \
\ \ \ \
(u,\eta)<(t+\e,\th)\\ \\
c'(s,(u,\eta)(t,\th)^{-1}), \ \ \ \ \ \ \ \ \ \
(t-\e,\th)<(u,\eta).
\end{array} \right.    \]

\noindent (Here $R_{t,\th}:S\times\ruu\ra S\times\ruu$ is the
right translation by $(t,\th)$ in the $\ruu$-direction, i.e.
$R_{t,\th}(s,(u,\eta))= (s, (u,\eta)(t,\th)).\ )$

\item  For any superpath $c:I_{t,\th}\ra M$, the bundle map
$SP(c):c_{0,0}^*E\ra c_{t,\th}^*E$ is an isomorphism, with inverse
given by $PS(\overline{c}):c_{t,\th}^*E\ra c_{0,0}^*E$, where
$\overline{c}:I_{t,\th}\ra M$ is given by $\overline{c}(u,\eta)=
c((u,\eta)^{-1}(t,\th))$, and, for a superpath $\a$ in $M$,
$PS(\a)$ denotes $Q$-parallel transport along $\a$ (see Remark
\ref{qtransport}).

\item (Invariance under reparametrization) Given $c:I_{t,\th}\ra
M$  a superpath in $M$ and $\f:I_{s,\eta}\ra I_{t,\th}$  a family
of diffeomorphisms of superintervals that preserve the vertical
distribution, we have $$SP(c\comp\f)= SP(c).$$

\end{enumerate}
Moreover, if $\na\neq\na'$ then $SP(\na)\neq SP(\na')$.
\end{thm}

\noindent\textbf{Proof of (2). }Since the construction of parallel
transport is natural in $S$ (see Lemma \ref{naturality}), it is
enough to consider the case when $S$ is ``small"  and $c$ and $c'$
map to a trivializing neighborhood $U\subseteq M$ for $E$, such
that $E_{|U}\cong U\times\rpq$ and $\na=d+ A$. If $\psi$ is a
super-parallel section along $c$ with $\psi(0,0)=\psi_0$ and
$\psi'$ is parallel along $c'$ with $\psi'(0,0)=\psi(t,\th)$ then
$\psi'\cdot\psi$ defined by
\[ \psi'\cdot \psi(s,u,\eta)=
\left\{
\begin{array}{l}  \ \ \ \ \ \psi(s,u,\eta),\ \ \ \ \ \ \ \ \ \ \ \ \
\ \ \ \
(u,\eta)<(t+\e,\th)\\ \\
\psi'(s,(u,\eta)(t,\th)^{-1}), \ \ \ \ \ \ \ \ \ \
(t-\e,\th)<(u,\eta)
\end{array} \right.    \]

is a parallel section along $c'\cdot c$. (Observe that
$\psi'\cdot\psi$ is well defined, by Prop. \ref{lift}.) To show
this, it is enough to prove the following
\begin{lem}
Let $c:S\times\ruu\ra M$ be a superpath in $M$, and $A\in
\O^1(M)\tensor End(\rpq)^{ev}$. Let also $\psi:S\times\ruu\ra
\rpq$ be such that
\[ \partial_D\psi+(c^*A)(D)\psi= 0. \]
If $\bar{c}= c\comp R_{(t,\theta)}$ and $\bar{\psi}= \psi\comp
R_{(t,\theta)}$,

then
\[  \partial_D\bar{\psi}+(\bar{c}^*A)(D)\bar{\psi}= 0. \]
\end{lem}

\begin{proof}
Let $R$ be short for $R_{(t,\theta)}$. Then $R^\sh$ extends to
$\rpq$-valued functions. Moreover, the vector field $D$ is
invariant under right translations, i.e. $R_*D= D$, or, written
differently, $D\comp R^\sh= R^\sh\comp D$. Applied to the
$\rpq$-valued function $\psi$, this gives
\[ \partial_D(\psi\comp R)= (\partial_D\psi)\comp R. \]
On the other hand,
\begin{eqnarray*}
(\bar{c}^*A)(D) &=& (R^*(c^*A))(D)\\
&=& R^\sh(c^*A(R_*D))\\
&=& R^\sh(c^*A(D))\\
&=& (c^*A(D))\comp R.
\end{eqnarray*}
Therefore we have
\begin{eqnarray*}
\partial_D\bar{\psi}+(\bar{c}^*A)(D)\bar{\psi} &=&
\partial_D(\psi\comp R)+R^*(c^*A)(D)(\psi\comp R)\\
&=& \partial_D(\psi)\comp R+\{(c^*A)(D)\psi\}\comp R\\
&=& 0.
\end{eqnarray*}

\end{proof}

\noindent\textbf{Proof of (3). }Again, it is enough to assume that
$c$ maps to a trivializing neighborhood, as before. Then $\psi$ is
parallel along $c$ if
\[ \partial_D\psi + (c^*A)(D)\psi= 0. \]
Consider the section $\bar{\psi}$ along $\bar{c}$ defined by
$\bar{\psi}(s,u,\eta)=\psi(s,(u,\eta)^{-1}(t,\th))$. Then
$\bar{\psi}$ is $Q$-parallel along $\bar{c}$. To see this it is
enough to prove the following
\begin{lem}
Let $c:S\times\ruu\ra M$ be a superpath in $M$, and $A\in
\O^1(M)\tensor End(\rpq)^{ev}$. Let also $\psi:S\times\ruu\ra
\rpq$ be such that
\[ \partial_D\psi+(c^*A)(D)\psi= 0. \]
If $\bar{c}= c\comp R_{(t,\theta)}\comp I$ and $\bar{\psi}=
\psi\comp R_{(t,\theta)}\comp I$, where $I:S\times\ruu\ra
S\times\ruu: (s, u, \eta)\mapsto (s,-u,-\eta)$ is the inversion
map, then
\[  \partial_Q\bar{\psi}+(\bar{c}^*A)(Q)\bar{\psi}= 0. \]
\end{lem}

\begin{proof}
Let us begin by showing that, via the inversion map
$I:\ruu\ra\ruu:(t,\th)\mapsto (-t,-\th)$, we have
\[ I_*D= -Q. \]
For that, we need to show that the following diagram is
commutative

\[ \xymatrix{\cinf(\ruu) \ar[r] ^{I^\sharp} \ar@{-->}[d]_{-Q=-(\vd)} & \cinf(\ruu) \ar[d]^{D=\vq}\\
\cinf(\ruu) \ar[r] ^{I^\sharp} & \cinf(\ruu).} \] Following the
diagram both ways we have

\[ \xymatrix{t \ar@{|->}[r] ^{I^\sharp} \ar@{|->}[d]_{-Q} & -t \ar@{|->}[d]^D & &
\th \ar@{|->}[r] ^{I^\sharp} \ar@{|->}[d]_{-Q}
& -\th \ar@{|->}[d]^D  \\
\th \ar@{|->}[r] ^-{I^\sharp} & -\th & & -1 \ar@{|->}[r]
^-{I^\sharp} & -1 .} \]
\newline
\noindent Coming back to the proof of the lemma, let us notice
that $\bar{\psi}$ can be written
\[ \bar{\psi}= I^\sh R^\sh\psi, \]
where $R$ is short for $R_{(t,\theta)}$. Then
\begin{eqnarray*}
\partial_Q(\bar{\psi}) &=& \partial_Q(I^\sh R^\sh\psi)\\
&=& -I^\sh\partial_D R^\sh\psi\\
&=& -I^\sh R^\sh\partial_D\psi,
\end{eqnarray*}
where the second equality is true by $\partial_Q I^\sh= -I^\sh
\partial_D$ above and the third equality is true since $D$ is a
right-invariant vector field, i.e. $R^\sh \partial_D= \partial_D
R^\sh$.\\

\noindent On the other side, we have
\begin{eqnarray*}
(\bar{c}^*A)(Q) &=& (I^*R^*c^*A)(Q)\\
&=& I^\sh(R^*(c^*A)(I_*(Q))) \\
&=& -I^\sh(R^*(c^*A)(D))  \\
&=& -I^\sh R^\sh((c^*A)(D)),
\end{eqnarray*}
where we used that the fact that
\[ (f^*\o)(Y)= f^\sharp(\o(f_*Y)), \]
with $f:N\ra M$ an arbitrary map of supermanifolds, $\o\in
\O^1(M)$, and $Y\in\sX(N)$. (The relation is true provided $f_*Y$
exists, which is true in our cases.) Therefore
\begin{eqnarray*}
\partial_Q\bar{\psi}+(\bar{c}^*A)(Q)\bar{\psi} &=& -I^\sh R^\sh\partial_D\psi-
I^\sh R^\sh((c^*A)(D))(I^\sh R^\sh\psi)\\
&=& -(\partial_D\psi+(c^*A)(D)\psi)\comp R\comp I                  \\
&=& 0.
\end{eqnarray*}
The lemma is proved.

\end{proof}

\noindent The conclusion of (3) follows.\\

\subsection{The parallel transport of $(\nabla, A)$}\label{adiabatic}
In the end of this section we define a notion of
\emph{$A$-parallel transport} for the pair consisting of a
connection and a bundle endomorphism $A$, and see that it
converges (by an ``inverse adiabatic limit" process) to the
parallel transport of the connection. In particular, this means
that the $A$-parallel transport is reparametrization invariant in
the limit.

Let $E$ be a super vector bundle over a supermanifold $M$. Let
$(\na, A)$ be a pair consisting of a (grading preserving)
connection $\na$ on $E$ and $A\in\G(M, End\ E)$ an \emph{odd}
endomorphism of $E$. Let $c:S\times \ruu\ra M$ in $M$ be a family
of supercurves parametrized by a supermanifold $S$. A section
$\psi\in\G(c^*E)$ of $E$ along $c$ is \emph{$A$-parallel} if it
satisfies the equation
\begin{equation} \label{par}
(c^*\na)_{D}\psi-(c^*A)\psi= 0.
\end{equation}

This is again a ``half-order" differential equation. In local
coordinates, if $E_{|U}\iso U\times\rpq$, then $\na= d+a$, with
$a\in\O^1(M, End E)^{odd}$ and the equation (\ref{par}) can be
written
\[ \partial_D\psi+ (c^*a)(D)\psi-(c^* A)\psi= 0, \] where $D=\partial_\eta+\eta\partial_u$.
Suppose for simplicity that $(u, \eta)$ runs on the superinterval
$I_{(T,\tau)}$, for $(T,\tau)\in \ruu_+(S)$ an $S$-superpoint of
$\ruu_+$. Recall that $I_{(T,\tau)}$ is defined by the embeddings
\[ \xymatrix{ S \ar@{^{(}->}[r]_-{i_{(0,0)}} & S\times\ruu & S
\ar@{_{(}->}[l]^-{i_{(T,\tau)}}}. \]

For $\l>0$, let
$$\f_\l:I_{(\l
T,\sqrt{\l}\tau)}\ra  I_{(T,\tau)}:\ (t,\th)\mapsto
(\frac{1}{\l}t, \frac{1}{\sqrt{\l}}\th) $$ be the ``rescaling"
diffeomorphism that preserves the distribution $\sD$. Then
$\tilde{\psi}$ is $A$-parallel with respect to $\ti{c}=
c\comp\f_\l$ if
\[ \partial_{\ti{D}}\ti{\psi}+ (c^*a)(\ti{D})\ti{\psi} -(\ti{c}^* A)\ti{\psi} = 0, \]
where $\ti{D}=\partial_\th+\th
\partial_t$. If we write $\ti{\psi}= \psi^\l\comp\f_\l= \f_\l^\sh(\psi^\l)$ then the
last equation can be rewritten
\[ \partial_{\ti{D}}(\f_\l^\sh(\psi^\l))+
\f_\l^*(c^*a)(\ti{D})\f_\l^\sh(\psi^\l)- \f_\l^\sh (c^*
A)\f_\l^\sh(\psi^\l)= 0,
\]
An easy calculation shows that
$\f_{\l*}(\ti{D})=\frac{1}{\sqrt{\l}}D$ 
which can be rewritten as $\partial_{\ti{D}}\f_\l^\sh=
\frac{1}{\sqrt{\l}}\f_\l^\sh\partial_D$, and the last equation is
equivalent to
\[ \frac{1}{\sqrt{\l}}\f_\l^\sh\partial_D\psi^\l+
\frac{1}{\sqrt{\l}}\f_\l^\sh((c^*a)(D))\f_\l^\sh(\psi^\l)-\f^\sh_\l(
(c^* A)\psi^\l)= 0, \] therefore
\[ \partial_D\psi^\l+ (c^*a)(D)\psi^\l -\sqrt{\l}(c^* A)\psi^\l = 0. \]

If we let $\l\to 0$ we see that $\psi^\l\longrightarrow \psi^0$,
where $\psi^0$ is the parallel section along $c$ determined by the
connection $\na$. We conclude that the parallel transport defined
by $(\na, A)$ converges in the ``inverse adiabatic limit" to the
parallel transport of $\na$, which is in particular invariant
under reparametrization. Symbolically we write
\[ SP(\na, A) \longrightarrow SP(\na). \]

\section{Superconnections and Parallel Transport} In this
section we prove our main result: Any superconnection $\A$ on a
$\bZ/2$-graded vector bundle over a manifold gives rise to a
parallel transport $SP(\A)$ which converges to the parallel
transport $SP(\A_1)$ determined by $\A_1$, the connection part of
the superconnection.

\subsection{Preliminaries} \label{prel} Start with a $\bZ/2$-graded vector
bundle $E$ over a \emph{manifold} $M$, and consider a
grading-preserving connection $\na$ on $E$, together with an $End\
E$- valued form $A$ on $M,\ A\in (\O^*(M, End\ E))^{odd}$.
Combining these two pieces, we obtain a Quillen connection $\A=
\na+ A$ on $E$.

Recall the identification in Section \ref{ident}.
\[ \underline{\sman}(\ruu, M) =\ \underline{\sman}(\bR, \Pi TM),
\] which for a supermanifold $S$ gives
\[ \sman(S\times\ruu, M)\cong\ \sman(S\times\bR, \Pi TM). \]

Let $c:\ S\times\ruu \ra M$ be a supercurve in $M$. Lift it to a
supercurve $\tilde{c}$ in $\Pi TM$  as follows
\[ \xymatrix{ \Pi TM\times \rou \ar[r]^-T & \Pi TM  \ar[d]^\pi \\
S\times \ruu \ar[r]_-c \ar[u]^{\hat{c}\times 1}
\ar@{-->}[ur]^{\tilde{c}} & M. } \] In other words, $\ti{c}=
T\comp (\hat{c}\times 1)$, where $T$ is the $\rou$ action map on
$\Pi TM$ (see Lemma \ref{equiv}). The map $\hat{c}:S\times\bR\ra
\Pi TM$ corresponds to $c:S\times\ruu\ra M$ under the above
identification. The map $\pi$ is given on functions by
$\pi^\sharp: \cinf(M) \ra \cinf(\Pi TM)=\ \O(M)$ the inclusion of
functions on $M$ into the space of differential
forms on $M$. \\\\
\emph{Claim: The above diagram is commutative}.\\\\
\emph{Proof of Claim:} It is enough to show that the following
diagram is commutative
\begin{equation} \label{dia}
\xymatrix{  \Pi TM\times \rou \ar[r]^-T & \Pi TM  \ar[d]^\pi \\
S\times \rou \ar[r]_-\a \ar[u]^{\hat{\a}\times 1} & M, }
\end{equation}
for $S$ an arbitrary supermanifold, and $\a:S\times\rou\ra M$ an
arbitrary map. Here $\hat{\a}:S\ra\Pi TM$ corresponds to $\a$ via
Lemma \ref{otb}.

This translates into the following diagram being commutative
\[ \xymatrix{ \O^*(M)[\th] \ar[d]_{\hat{\a}^\sh\tensor 1} & \O^*(M)
\ar[l]_-{T^\sh} \\
\cinf(S)[\th] & \cinf(M)\ar[l]^{\a^\sh} \ar[u]_{\pi^\sh}. } \]

\noindent Recall (see the proof of Lemma \ref{otb}) that if
$\a^\sh: f\mapsto a^\sh(f)+\th X_a(f)$, for $a\in M(S)$ and
$X_a\in TM_a$, then $\hat{\a}^\sh: \O^*(M)\ra\cinf(S)[\th]$ is
determined by saying that $f\mapsto a^\sh(f),\ df\mapsto X_a(f)$.
Therefore, we have
\[ f\over{\pi^\sh}\longmapsto f \over{T^\sh} \longmapsto f+\th df
\over{\hat{\a}^\sh\tensor 1} \longmapsto a^\sh(f)+\th X_a(f)=
\a^\sh(f).\] To complete the proof of the claim, it is enough to
replace $S\mapsto S\times\bR$ and $\a\mapsto c$ in the above
considerations.

\begin{rem} \label{nat} It is not hard to check that the construction
$c\mapsto \ti{c}$ is natural in $S$, i.e. $$\widetilde{c\comp
(\f\times 1)}= \ti{c}\comp(\f\times 1),$$ for $\f:S'\ra S$ an
arbitrary map of supermanifolds.
\end{rem}

Given the supercurve $c$, consider the pull-back diagram
\[ \xymatrix{ E \ar[dd] & & c^*E \ar[ll] \ar[dl] \ar[dd] \\
& \pi^*E \ar[ul] \ar[dd] & \\
M & & S\times\ruu \ar'[l][ll]_{\ \ \ \ c} \ar@{-->}[dl]^-{\tilde{c}} \\
& \Pi TM \ar[ul]^\pi &  } \] where $\tilde{c}$ is as above. We
call a section $\psi\in\G(c^*E)$ of $E$ along $c$
\emph{$\A$-parallel} if it satisfies the equation
\[  (c^*\na)_{D}\psi- (\tilde{c}^*A)\psi= 0.  \]
This is again a ``half-order" differential equation. It is
equivalent to the equation
\[  (\ti{c}^*(\pi^*\na))_{D}\psi- (\tilde{c}^*A)\psi= 0.  \]

Therefore $\psi\in \G(c^*E)$ is $\A$-parallel if and only if
$\psi$ is $A$-parallel along the \emph{lift} $\ti{c}$ with respect
to the pair $(\pi^*\na, A\in \G(End(\pi^*E)))$ on the bundle
$\pi^*E\ra \Pi TM$, as defined in Section \ref{adiabatic}.
Therefore Prop. \ref{lift} gives the following

\begin{prop}
Let $c:S\times\ruu \ra M$ be a supercurve in the compact manifold
$M$. Let $\psi_0\in\G(c_{0,0}^*E)$ be a section of $E$ along
$c_{0,0}:S\ra M$. Then, there exists a unique $\A$-parallel
section $\psi$ of $E$ along $c$, such that $\psi(0,0)= \psi_0$.
\end{prop}

\subsection{Parallel transport along superpaths} Let $c:I_{t,\th}
\ra M$ be a superpath in $M$ with $c(0,0)=x$ and $c(t,\th)=y$.
Then the $\A$-parallel transport of $\na$ and $A\in\O^*(M, End E)$
will determine a bundle homomorphism $SP(c):x^*E\ra y^*E$. This is
defined as in Section \ref{spath} by a $\cinf(S)$-linear map
$SP(c): \G(S, x^*E)\ra \G(S, y^*E)$ described by the diagram
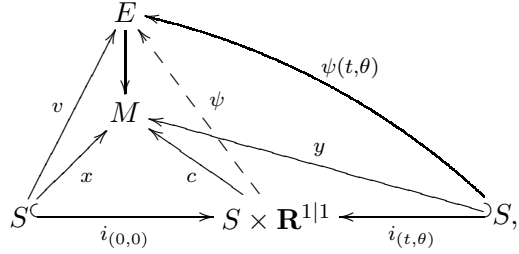
\[ \xymatrix{ & E \ar[d] & & \\
& M & & \\
S \ar[ur]_x \ar[uur]^v \ar@{^{(}->}[rr]_-{i_{(0,0)}} & &
S\times\ruu \ar[ul]^c \ar@{-->}[uul]_\psi & & S
\ar@{_{(}->}[ll]^-{i_{(t,\th)}} \ar[ulll]_y
\ar@/_1pc/[uulll]_{\psi(t,\th)}, }  \] i.e. $SP(c)(v)= \psi(t,
\th)$, where $\psi$ is the unique $\A$-parallel section with
respect to the pair $(\na, A)$ of $E$ along the supercurve $c$,
such that $\psi\comp i_{(0,0)}= v$.

\subsection{Main Theorem} We are now in the position to state our
main theorem.

\begin{thm}
Let $E$ be a $\bZ/2$-graded vector bundle over a manifold $M$. Let
$\na$ be a grading preserving connection on $E$ and $A\in\O^*(M,
End E)^{odd}$. The pair $(\na, A)$ gives rise to a correspondence
$SP= SP(\na, A)$
\[ \xymatrix{ I_{t,\th} \ar[r]^c & M & \ar@{|->}[rr]^-{SP} & & &
c_{0,0}^*E \ar[r] & c_{t,\th}^*E} \] such that:
\begin{enumerate}
\item The correspondence $c\mapsto SP(c)$ is smooth, and natural
in $S$ (see Lemma \ref{naturality}). \item (Compatibility under
glueing) If $c:I_{t,\th}\ra M$ and $c':I_{t',\th'}\ra M$ are two
superpaths in $M$ such that $c'\equiv c\comp R_{t,\th}$ on some
neighborhood $S\times U$ of $S\times (0,0)\hookrightarrow
S\times\ruu$, with $U$ an open subsupermanifold in $\ruu$
containing $(0,0)$, we have
\[ SP(c'\cdot c)= SP(c')\comp SP(c), \]
where $c'\cdot c:I_{t'+t+\th'\th, \th'+\th}\ra M$ is obtained from
$c$ and $c'$ by glueing them along their ``common endpoint".
\end{enumerate}
Moreover, if $\na\neq\na'$ or $A\neq A'$ then $SP(\na, A)\neq
SP(\na', A')$. Also, $SP(\na, A)$ converges in the inverse
adiabatic limit to $SP(\na)$.

\begin{proof} The properties (1) and (2) are clear from the
construction of the parallel transport of the pair $(\na,A)$. Two
different such pairs (superconnections) give rise to two different
parallel transports, since the parallel transport recovers the
superconnection, as we will show in the following Section
\ref{recover}. The inverse adiabatic limit process is described in
Section \ref{adiabatic}.
\end{proof}

\end{thm}

\begin{cor} \textbf{(The parallel transport of a superconnection)}
A superconnection $\A$ on the bundle $E$ over $M$ (in the sense of
Quillen) gives rise in a unique way to a (super) parallel
transport based on $M$, namely consider the parallel transport
$SP(\na, A)$ associated to the pair $(\na=\A_1, A=\sum_{i\neq
1}\A_i)$.
\end{cor}

\subsection{Recovering the superconnection} \label{recover} In this
section we show how to recover a superconnection- i.e.  a pair
$(\na, A)$ - from the parallel transport associated to it. We have
already seen in Section \ref{adiabatic} how the parallel transport
of $(\na, A)$ converges via an inverse adiabatic limit process to
the parallel transport of $\na$, which further recovers the
connection $\na$ - see Section \ref{recover}. We are only left
with obtaining $A\in \O^*(M, End\ E)$. To do that, let us consider
the following diagram

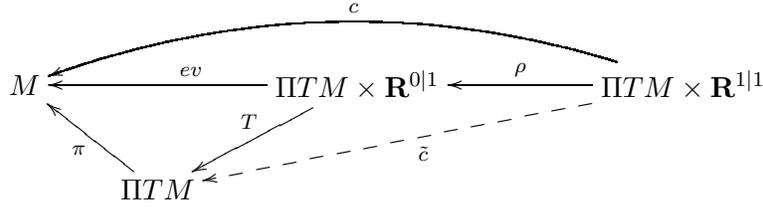
\[ \xymatrix{ M & & \Pi TM\times\rou \ar[ll]_{ev} \ar[dl]_T & & \Pi
TM\times\ruu \ar[ll]_\rho \ar@{-->}[dlll]^{\ti{c}} \ar@/_2pc/[llll]_c \\
& \Pi TM \ar[ul]^\pi & & & } \] where $ev$ is the ``evaluation"
map as in the previous section and $\rho= 1_{\Pi TM} \times p$,
with $p:\ruu\ra \rou$ the natural projection map. Let us remark
first that the lift of the curve $c= ev\comp \rho$ is the
composition $T\comp\rho$. This holds by the naturality of lifts of
supercurves  (see Remark \ref{nat}), and the fact that the lift of
the ``curve" $ev$ is given by $T$ (in diagram \ref{dia}, if $\a=
ev$ then $\widehat{ev}= 1_{\Pi TM})$.

By definition, a section $\psi\in \G(c^*E)$ of $E$ along $c$ is
$\A$-parallel if
\[ (c^*\na)_D\psi- (\rho^*T^*A)\psi= 0. \]
We therefore know the operator
\[ (c^*\na)_D- \rho^*T^*A:\  \G(c^*E)\lra \G(c^*E) \]
on parallel sections. But that is enough to determine it, since
the parallel sections generate $\G(c^*E)$ as a $\cinf(\Pi TM\times
\ruu$)- module. On the other hand, we know the operator
$(c^*\na)_D: \G(c^*E)\lra \G(c^*E)$, since we know the connection
$\na$. In this manner we determine the linear map $\rho^*T^*A$.
Since both $\rho^*$ and $T^*$ are injective, this uniquely
determines $A$. In this manner, we recovered the superconnection
$(\na, A)$ from the associated parallel transport.

\subsection{An example. } Let us conclude by considering the above construction
in the case of $M=\ pt$. The bundle $E$ together with the
connection reduces in this case to a $\bZ/2$-graded vector space
$V$, and the bundle endomorphism valued form $A$ reduces to an odd
endomorphism $A\in End^1(V)$. We have the pull-back diagram
\[ \xymatrix{ E=E \ar[d] & E\times \ruu \ar[d] \ar[l] \\
pt=\ pt & \ruu \ar[l]^-{c=\tilde{c}}. } \] (It's enough to
consider only this map, since the factor $S$ doesn't play a role
here.) The  pull-back bundle is endowed with the trivial
connection. The super parallel sections along $c$ are therefore
given by the equation
\[ D\psi= A\psi. \]
\begin{lem} The solutions of the above equation  are given by
$$(t,\th)\mapsto e^{-tA^2+\th A}v,$$ for some $v$ in $V$.
\end{lem}

\begin{proof} Indeed, we have:
\begin{eqnarray*}
(\partial_\th+ \th\partial_t)e^{-tA^2+\th A} &=& (\partial_\th+
\th\partial_t)[(1+\th A)e^{-tA^2}] \\
&=& Ae^{-tA^2}+ \th e^{-tA^2}(-A^2)\\
&=& A(1+ \th A)e^{-tA^2} \\
&=& Ae^{-tA^2+\th A}
\end{eqnarray*}
where in the third equality we moved $A$ past $e^{-tA^2}$ without
a sign change since $e^{-tA^2}$ is even, and past $\th$ with a
change of sign, since both $A$ and $\th$ are odd. The lemma
follows.
\end{proof}

The parallel transport therefore defines a map
\[ \ruu \ni(t, \th)\mapsto e^{-tA^2+\th A}\in GL(V), \]
which is in fact a \emph{supergroup homomorphism} $\ruu\ra GL(V)$,
since composition on $\ruu$, which preserves the vector field $D$,
corresponds to composition (multiplication) on $GL(V)$. For a
direct proof of this see \cite{ST}.

\bibliographystyle{plain}
\bibliography{bibliografie}

\end{document}